\documentclass{amsart}
\usepackage{amsmath}
\usepackage{amssymb, amsthm, amsfonts, tikz-cd, mathrsfs,mathtools,stmaryrd,enumitem, algorithmic,float,comment,tikz}
\usetikzlibrary{backgrounds}
\usetikzlibrary{decorations.pathreplacing}
\usepackage{fullpage}
\usepackage{xcolor}
\usepackage[ruled,vlined,linesnumbered]{algorithm2e}

\usepackage[colorlinks,citecolor=blue,linkcolor=blue,urlcolor=red]{hyperref}

\usepackage[capitalize,nameinlink]{cleveref}
\crefname{theorem}{Theorem}{Theorems}
\crefname{lemma}{Lemma}{Lemmas}
\crefname{claim}{Claim}{Claims}
\crefname{prop}{Proposition}{Propositions}

\newtheorem{theorem}{Theorem}[section]
\newtheorem{lemma}[theorem]{Lemma}
\newtheorem{prop}[theorem]{Proposition}
\newtheorem{cor}[theorem]{Corollary}

\newtheorem{definition}[theorem]{Definition}

\newtheorem{claim}[theorem]{Claim}

\crefalias{step}{enumi}
\crefname{step}{Step}{Steps}

\crefalias{property}{enumi}
\crefname{property}{Property}{Properties}
\crefname{figure}{Figure}{Figures}

\newcommand{\R}{\mathbb{R}}

\newcommand{\E}{\mathbb{E}}

\newcommand{\lpr}[1]{\left(#1\right)}
\newcommand{\lbr}[1]{\left[#1\right]}

\newcommand{\ceil}[1]{\lceil#1\rceil}

\newcommand{\ang}[1]{\langle#1\rangle}
\newcommand{\labs}[1]{\left|#1\right|}

\newcommand{\eps}{\varepsilon}
\newcommand{\supp}{\operatorname{supp}}
\newcommand{\signbin}{\{\pm 1\}}

\usepackage[
backend=biber,
style=numeric,
maxnames=99
]{biblatex}
\renewbibmacro{in:}{%
  \ifentrytype{article}{}{\printtext{\bibstring{in}\intitlepunct}}}

\title{Essential Covers of the Hypercube require many Hyperplanes}

\author{Lisa Sauermann}
\thanks{University of Bonn. \texttt{sauermann@iam.uni-bonn.de}. Research supported in part by NSF Award DMS-2100157 and a Sloan Research Fellowship.}

\author{Zixuan Xu}
\thanks{Massachusetts Institute of Technology. \texttt{zixuanxu@mit.edu}}

\begin{document}

\begin{abstract}
We prove a new lower bound for the almost 20 year old problem of determining the smallest possible size of an essential cover of the $n$-dimensional hypercube $\{\pm 1\}^n$, i.e.\ the smallest possible size of a collection of hyperplanes that forms a minimal cover of $\{\pm 1\}^n$ and such that furthermore every variable appears with a non-zero coefficient in at least one of the hyperplane equations. We show that such an essential cover must consist of at least $10^{-2}\cdot n^{2/3}/(\log n)^{2/3}$ hyperplanes, improving previous lower bounds of Linial--Radhakrishnan, of Yehuda--Yehudayoff and of Araujo--Balogh--Mattos.
\end{abstract}

\maketitle

\section{Introduction}

There is a long line of research, spanning over three decades, on problems about covering the vertices of the $n$-dimensional hypercube $\signbin^n$ by hyperplanes (see e.g. \cite{alon-furedi-93,araujo2022,clifton-huang-20,klein2022slicing,LinialR05,YehudaY21}). A very simple question is how many hyperplanes are needed in order to cover all vertices in $\signbin^n$, i.e.\ such that every vertex is contained in at least one of these hyperplanes. This question has an equally simple answer: two hyperplanes are enough, for example one can take the two hyperplanes given by the equations $x_1=1$ and $x_1=-1$. However, this hyperplane cover is not truly $n$-dimensional, in the sense that the variables $x_2,\dots,x_n$ do not appear in any of the hyperplane equations. It is therefore natural to demand in addition that each variable has a non-zero coefficient in the equation for at least one of the hyperplanes.

However, this still does not lead to a very interesting problem. In addition to the hyperplanes with equations $x_1=1$ and $x_1=-1$ that already cover all vertices, one could add an additional hyperplane with an equation like $x_1+x_2+\dots+x_n=0$ containing all variables with non-zero coefficients. But in some sense, this additional hyperplane is not truly part of the hyperplane cover, as it is not needed for covering all vertices (the other two hyperplanes together already cover all vertices in $\signbin^n$).

This leads to the following notion of an \emph{essential cover} of the hypercube $\signbin^n$, introduced by Linial and Radhakrishnan \cite{LinialR05} in 2005. It describes a minimal set of hyperplanes covering all vertices in $\signbin^n$, which in addition has the property that each variable appears in the equation of at least one of the hyperplanes. Formally, an essential cover of the hypercube $\signbin^n$ is defined as follows.

\begin{definition}[Essential cover]\label{def:essential-cover}
A collection of hyperplanes  $h_1,\dots,h_k$ in $\mathbb{R}^n$  is called an \emph{essential cover} of the $n$-dimensional hypercube $\signbin^n$ if the following conditions are satisfied:
\begin{enumerate}[label=(E{{\arabic*}})]
    \item For every vertex $x\in \signbin^n$, there is some $i\in [k]$ such that $x\in h_i$.
    \item For every $j\in [n]$, there is some $i\in [k]$ such that, writing $h_i=\{x\in \R^n \mid \ang{v_i,x} = \mu_i\}$, we have $v_{ij}\ne 0$. 
    \item For every $i\in [k]$, there is a vertex $x\in \signbin^n$ that is covered only by $h_i$ and not by any of the the hyperplanes $h_1,\dots,h_{i-1},h_{i+1},\dots,h_n$.
\end{enumerate}
\end{definition}

Condition (E1) means that every vertex in $\signbin^n$ is covered by some hyperplane, and condition (E3) means that $h_1,\dots,h_k$ is a minimal collection of hyperplanes with respect to the property of covering every vertex. Finally, condition (E2) can be restated as saying that each of the variables $x_1,\dots,x_n$ must appear with a non-zero coefficient in at least one of the hyperplane equations $\ang{v_i,x} = \mu_i$ for $i=1,\dots,n$ (where $v_i=(v_{i1},\dots,v_{in})\in \mathbb{R}^n$ and $\mu_i\in \mathbb{R}$, so $\ang{v_i,x} = \mu_i$ is a shorter form of writing the linear equation $v_{i1}x_1+\dots+v_{in}x_n= \mu_i$). Note that this condition is not affected by rescaling the equations $\ang{v_i,x} = \mu_i$, so it does not depend on the parametrization chosen for each of the hyperplanes.

It is now a very natural question to ask how large an essential cover of $\signbin^n$ needs to be. Linial and Radhakrishnan \cite{LinialR05} showed in 2005 that any essential cover must contain at least $\Omega(\sqrt{n})$ hyperplanes. On the other hand, they also gave a construction of an essential cover consisting of $\ceil{n/2}+1$ hyperplanes, which still remains the best known upper bound. Yehuda and Yehudayoff \cite{YehudaY21} improved the lower bound to $\Omega(n^{0.52})$, and more recently Araujo, Balogh, and Mattos \cite{araujo2022} further improved the lower bound to $\Omega(n^{5/9}/(\log n)^{4/9})$ by refining the methods in \cite{YehudaY21}.

In this paper, we further improve the lower bound for the size of an essential cover of the hypercube.

\begin{theorem}\label{thm:bound}
    For $n\ge 2$, any essential cover of the $n$-dimensional hypercube $\signbin^n$ must consist of at least $10^{-2}\cdot n^{2/3}/(\log n)^{2/3}$ hyperplanes.
\end{theorem}

The absolute constant $10^{-2}$ is not optimized in our proof.

We remark that our lower bound for the size of essential covers also implies new lower bounds for some problems in proof complexity, see the discussion in \cite{YehudaY21}.

Linial and Radhakrishnan \cite{LinialR05}, obtained their  lower bound of $\Omega(\sqrt{n})$ for the size of essential covers of $\signbin^n$ by showing the following. For any essential cover of $\signbin^n$ consisting of $k$ hyperplanes, each hyperplane equation contains at most $2k$ variables with non-zero coefficients. We also use this fact in our proof of \cref{thm:bound}.

We furthermore observed that this fact leads to an essentially tight lower bound for another well-known problem concerning hyperplane covers of $\signbin^n$, which can be stated as follows. A \emph{skew cover} of the $n$-dimensional hypercube $\signbin^n$ is a collection of hyperplanes $h_1,\dots,h_k$ covering all vertices in $\signbin^n$ such that each of the hyperplane equations contains all of the variables $x_1,\dots,x_n$ with non-zero coefficients (i.e. for $i=1,\dots,n$, when writing $h_i=\{x\in \R^n \mid \ang{v_i,x} = \mu_i\}$, the vector $v_i$ has full support). The problem is again to determine the smallest possible size of a skew cover of $\signbin^n$. The previous best known lower bound for this problem was $\Omega(n^{2/3}/(\log n)^{4/3})$, which follows from a recent result of Klein \cite{klein2022slicing} (via the same argument as in \cite{yehuda2021slicing}). From the above-mentioned fact about essential covers proved by Linial and Radhakrishnan \cite{LinialR05}, one can deduce the following stronger lower bound for the size of skew covers of $\signbin^n$.

\begin{prop}\label{prop:skew-cover}
For $n > 0$, any skew cover of the $n$-dimensional hypercube $\signbin^n$ must consist of at least $n/2$ hyperplanes.
\end{prop}

The proof of this proposition is given in \cref{sec:prelim}. The bound is tight up to constant factors, as it is not hard to construct a skew cover of $\signbin^n$ consisting of $n+1$ hyperplanes (by taking the hyperplanes described by the equations $x_1+\dots+x_n=n-2t$ for $t=0,\dots,n$).\\
    
\textit{Notation.} For a positive integer $n$, we use $[n]$ to denote the set $\{1,\dots, n\}$. Given sets $A, B, C$, we write $C = A\sqcup B$ to denote a partition of the set $C$ into two disjoint subsets $A$ and $B$. For any real $x>0$, we write $\log x$ for the natural logarithm of $x$ in base $e$. For two sets $A, B$, we use $B^A$ to denote the set of maps $A\to B$. For $x\in B^A$ and a subset $A'\subseteq A$, we denote the restriction of $x$ to $A'$ by $x|_{A'}$, as usual. Given $x\in B^A$, one may think of $x$ as a vector of length $|A|$ indexed by the elements in $A$ with each coordinate taking values in $B$, and for $a\in A$ we write $x_a$ for the value $x(a)\in B$. For convenience, given a positive integer $n$, we use $B^n$ to denote $B^{[n]}$ (of course, this is consistent with the usual definition $B^n=B\times B\times \dots\times B$). 

For a vector $v = (v_1,\dots, v_n)\in \R^n$, we use $\supp(v)$ to denote the support of $v$, which is the set of indices $i\in [n]$ where $v_i\ne 0$. We recall that the $\ell_1$-norm of $v$ is defined as $||v||_1 = \sum_{i = 1}^n |v_i|$; the $\ell_2$-norm of $v$ is defined as $||v||_2 = \lpr{\sum_{i = 1}^n v_i^2}^{1/2}$; and the $\ell_\infty$-norm of $v$ is defined as $||v||_\infty = \max_{i\in [n]}|v_i|$. Given two vectors $u,v\in \R^n$, we write $\ang{u,v} = \sum_{i = 1}^n u_iv_i$ for the standard inner product on $\R^n$. For an $n\times m$ matrix $V\in \R^{n\times m}$ and subsets $A\subseteq [n]$ and $ B\subseteq [m]$, we use $V[A\times B]$ to denote the $|A|\times |B|$ submatrix of $V$ consisting of the rows in $A$ and the columns in $B$.

\section{Preliminaries}\label{sec:prelim}

In this section, we discuss useful definitions and lemmas for our argument. We start with the following two definitions that characterize vectors with many different magnitudes. 

\begin{definition}[Magnitude]\label{def:magnitude}
   For any $x\in \R\setminus \{0\}$, we say that $x$ has \emph{magnitude} $j\in \mathbb{Z}$ if $10^{j} \le |x| < 10^{j+1}$.
\end{definition}

\begin{definition}\label{def:vector-magnitude}
    For a positive integer $S > 0$, we say that a vector $v$ has at least $S$ magnitudes if there exist $S$ non-zero entries in $v$ with distinct magnitudes.
\end{definition}

We remark that the notion of having many magnitudes is a simpler version of the notion of having ``many scales'' introduced in \cite{yehuda2021slicing} and later used in \cite{araujo2022,YehudaY21}.

The motivation behind the above definitions is the following lemma, which states that given a vector $v\in \R^n$ with at least $S$ magnitudes, for a random vector $w\in \signbin^n$, the probability of the event $\ang{v,w} = \alpha$ for any given value $\alpha\in \R$ is exponentially small in $S$. For our purposes, it is convenient to state the lemma for a biased random vector $w\in \signbin^n$. We remark that a similar lemma also appears in \cite{klein2022slicing}.

\begin{lemma}\label{lem:many-magnitudes}
Let $v\in \R^{n}$ be a vector with at least $S$ magnitudes. Let $p\in [1/3, 2/3]^n$ and consider a random vector $w\in \signbin^n$ with independent random entries $w_1,\dots,w_n\in \signbin$, whose distributions are given by
\[w_i = \begin{cases} 1 & \text{with probability } p_i \\ -1 & \text{with probability } 1-p_i \end{cases} \quad \text{for all } i\in [n].\]
Then for any $\alpha\in \R$, we have 
\[\Pr\Big[\ang{v, w} = \alpha\Big] \le \lpr{\frac{2}{3}}^{\ceil{S/2}}.\]
\end{lemma}

\begin{proof}
Let $r = \ceil{S/2}$. Among the at least $S$ different magnitudes occurring for the entries of the vector $v$, there are at least $r$ even numbers or at least $r$ odd numbers. So, upon relabeling the indices, we may assume without loss of generality that the first $r$ entries $v_1,\dots, v_r$ of $v=(v_1,\dots,v_n)\in \R^n$ have distinct magnitudes with the same parity and are ordered in decreasing order of magnitudes. Now, for each $i\in [r-1]$, if $m$ is the magnitude of $v_i$, then $v_{i+1}$ has magnitude at most $m-2$. This means that $|v_i|\ge 10^m$ and $|v_{i+1}| < 10^{m-1}$, and therefore $|v_i| \ge 10 |v_{i+1}|$ for every $i\in [r-1]$.

This implies that for any $\alpha\in \R$ there is at most one assignment for $w_1,\dots, w_r\in \signbin$ satisfying $\sum_{i = 1}^r v_iw_i = \alpha$.

Let us now condition on arbitrary outcomes of $w_{r+1},\dots, w_{n}\in \signbin$. Conditional on such fixed outcomes, it suffices to show that, for any $\alpha\in \R$, we have $\sum_{i = 1}^r v_iw_i = \alpha - \sum_{j = r+1}^n v_jw_j$ with probability at most $\lpr{\frac{2}{3}}^{r}$. Recall that there is at most one assignment for $w_1,\dots, w_r\in \signbin$ satisfying this equation. Since $w_1,\dots, w_r$ are independent random variables with
\[\Pr[w_i = 1] \le \frac{2}{3} \quad \text{and} \quad \Pr[w_i = -1] \le \frac{2}{3}\]
for every $i\in [r]$, we obtain
\[\Pr\lbr{\sum_{i = 1}^r v_iw_i = \alpha - \sum_{j = r+1}^n v_jw_j \, \Bigg| \, w_{r+1},\dots, w_n } \le \lpr{\frac{2}{3}}^r = \lpr{\frac{2}{3}}^{\ceil{S/2}}.\]
This implies our desired statement.
\end{proof}

The following lemma is due to Linial and Radhakrishnan \cite{LinialR05} and gives an upper bound for the number of non-zero coefficients for the hyperplane equations in an essential cover. It can be proved via the Combinatorial Nullstellensatz \cite{alon99}.

\begin{lemma}[\cite{LinialR05}]\label{lem:row-support}
Let $h_1,\dots,h_k$ be hyperplanes in $\R^n$ forming an essential cover of the hypercube $\signbin^n$. For each $i\in [k]$, let $h_i=\{x\in \R^n \mid \ang{v_i,x} = \mu_i\}$. Then we have $|\supp(v_i)|\le 2k$ for all $i\in [k]$.
\end{lemma}

We observed that \cref{lem:row-support} also gives a lower bound for the number of hyperplanes in any skew cover of the hypercube $\signbin^n$, as stated in \cref{prop:skew-cover}. Recall that the hyperplanes $h_1,\dots, h_k$ form a \emph{skew cover} of the hypercube $\signbin^n$ if every vertex in $\signbin^n$ is covered by at least one of the hyperplanes $h_1,\dots, h_k$ and for each $i=1,\dots,k$, when writing $h_i=\{x\in \R^n \mid \ang{v_i,x} = \mu_i\}$, all coordinates of vector $v_i\in \R^n$ are non-zero.

\begin{proof}[Proof of \cref{prop:skew-cover}]
Suppose $h_1,\dots, h_k$ are hyperplanes forming a skew cover of $\signbin^n$. For $i=1,\dots,k$, let us write $h_i=\{x\in \R^n \mid \ang{v_i,x} = \mu_i\}$. Then for every $i\in [k]$, all coordinates of the vector $v_i$ are non-zero.

Let $I\subseteq [k]$ be a minimal subset with respect to the property that the hyperplanes $h_i$ for $i\in I$ cover all vertices in $\signbin^n$ (and note that $I\ne \emptyset$). We claim that these hyperplanes $h_i$ for $i\in I$ form an essential cover of $\signbin^n$. Indeed, conditions (E1) and (E3) in \cref{def:essential-cover} are satisfied by the choice of $I$. Condition (E2) is satisfied, because for every index $i\in I$ all coordinates of $v_i$ are non-zero (so for any $j\in [n]$ one can take an arbitrary index $i\in I$ in condition (E2)).

Now, taking an arbitrary index $i\in I$, \cref{lem:row-support} implies that $n=|\supp(v_i)|\le 2|I|\le 2k$, so we must have $k\ge n/2$.
\end{proof}

Our proof of \cref{thm:bound} also relies on the following lemma, which Ball \cite{ball1991} extracted from Bang's solution of Tarski's plank problem \cite{bang1951} (see for example \cite{araujo2022} for a short self-contained proof of this lemma).

\begin{lemma}[Bang's lemma \cite{bang1951,ball1991}]\label{lem:bang}
Let $M$ be a  symmetric $\ell\times \ell$ matrix such that $M_{ii} = 1$ for every $i\in [\ell]$. Then for any vector $\mu = (\mu_1,\dots, \mu_\ell)\in \R^\ell$ and any real number $\theta \ge 0$, there exists $\eps\in \{\pm 1\}^\ell$ such that 
\[\big|(M(\theta\eps))_i - \mu_i\big|\ge \theta \quad \text{for all } i\in [\ell].\]
\end{lemma}

In our proof of \cref{thm:bound}, we apply the following corollary of \cref{lem:bang} to a certain submatrix of the coefficient matrix of an essential cover. This allows us to find a point that is far from certain hyperplanes in the cover.

\begin{cor}\label{cor:bang}
Let $V\in \R^{\ell\times m}$ be an $\ell\times m$ matrix with row vectors $v_1,\dots, v_\ell\in \R^{m}$ and column vectors $v_{*1},\dots,v_{*m}\in \R^\ell$. Suppose that $||v_i||_2 = 1$ for all $i\in [\ell]$, and let $\theta \ge 0$ be a real number satisfying $\theta ||v_{*j}||_1 \le 1/3$ for all $j\in [m]$. Then for any vector $\mu = (\mu_1,\dots, \mu_\ell)\in \R^\ell$, there exists $y\in \R^m$ such that $||y||_\infty \le 1/3$ and
\[\big|\ang{v_i, y} - \mu_i\big|\ge \theta \quad \text{for all } i\in [\ell].\]
\end{cor}

\begin{proof}
Apply \cref{lem:bang} to the symmetric $\ell\times 
\ell$ matrix $M:= V V^T$ and the vector $\mu\in \R^{\ell}$. Note that we have $M_{ii} =1 $ for all $i\in [\ell]$ since $||v_i||_2 = 1$ for all $i\in [\ell]$. Hence there exists a vector $\eps\in \signbin^\ell$ such that 
\[\big|(M(\theta\eps))_i - \mu_i\big|\ge \theta \quad \text{for all } i\in [\ell].\]
Setting $y := \theta V^T \eps\in \R^m$, for every $i\in [\ell]$ we obtain 
\[\big|\ang{v_i, y} - \mu_i\big| =\big|(Vy)_i - \mu_i\big|=\labs{(V(\theta V^T \eps))_i - \mu_i}=\labs{(VV^T(\theta  \eps))_i - \mu_i}=\labs{(M(\theta\eps))_i - \mu_i}\ge \theta.\]
Furthermore, we have 
\begin{align*}
    ||y||_\infty &= ||\theta V^T\eps||_\infty 
    = \max_{j\in [m]}\labs{\theta \sum_{i = 1}^\ell v_{ij}\eps_i}
    \le \max_{j\in [m]}\theta \sum_{i = 1}^\ell |v_{ij}| = \max_{j\in [m]}  \theta ||v_{*j}||_1  
    \le \frac{1}{3},
\end{align*}
as desired.
\end{proof}

Finally, we will use the following well-known concentration inequality.

\begin{lemma}[Hoeffding's inequality \cite{hoeffding63}]\label{lem:hoeffding}
    Let $a_1,\dots, a_\ell, b_1,\dots, b_\ell\in \R$ and let $z_1,\dots, z_\ell$ be independent real random variables such that for all $j\in [\ell]$, we always have $a_j\le z_j\le b_j$. Let $z = \sum_{i= 1}^\ell z_i$, then for every $t > 0$ we have
    \[\Pr\Big[\big|z - \E[z]\big| \ge t\Big] \le 2\cdot \exp\lpr{-\frac{2t^2}{\sum_{j = 1}^\ell (b_j - a_j)^2}}.\]
\end{lemma}

\section{Outline}
We employ a strategy similar to the one in \cite{YehudaY21}. Given an essential cover of the hypercube $\signbin^n$, we consider its coefficient matrix, recording the coefficients of the linear equations corresponding to the hyperplanes in the cover. More formally, if the hyperplanes are given by equations of the form $\ang{v_i,x} = \mu_i$ for $i=1,\dots,k$, then the coefficient matrix of the essential cover is the $k\times n$ matrix with rows $v_1,\dots,v_k$.

We first show that for any essential cover consisting of only a small number of hyperplanes, there is a decomposition of the coefficient matrix of a certain structured form. Using this decomposition, we then obtain a contradiction by probabilistically finding a vertex in $\signbin^n$ that is not covered by any of the hyperplanes in the cover. Our decomposition of the coefficient matrix is given by the following proposition.

\begin{prop}[Matrix decomposition]\label{prop:decomposition}
For $n\ge 2$, let $h_1,\dots,h_k$ be hyperplanes in $\R^n$ forming an essential cover of the hypercube $\signbin^n$. For each $i\in [k]$, let $h_i=\{x\in \R^n \mid \ang{v_i,x} = \mu_i\}$, and let $V\in \R^{k\times n}$ be the matrix with rows $v_1,\dots,v_k$. If $k\le 10^{-2}\cdot n^{2/3}/(\log n)^{2/3}$, then there exists a partition $[k] = K_1\sqcup K_2\sqcup K_3$ of the row indices of $V$ and a partition $[n] = N_1\sqcup N_2$ of the column indices of $V$ with $N_1 \ne \emptyset$, such that the following conditions hold:
\begin{enumerate}
    \item Every entry in the submatrix $V[K_1 \times N_1]$ is zero. \label[property]{item:K1N1zero}
    \item There exists a $|K_2|\times |N_1|$ matrix $V'$ that can be obtained from $V[K_2\times N_1]$ by rescaling the rows in  in some way (by some non-zero real numbers), such that every row in $V'$ has $\ell_2$-norm equal to $1$ and every column in $V'$ has $\ell_1$-norm at most $(60\log n)^{-1/2}$.\label[property]{item:norm}
    \item In the submatrix $V[K_3\times N_1]$, every row has at least $\ceil{10\log n}$ magnitudes.\label[property]{item:many-magnitudes}    
\end{enumerate}

\end{prop}

\begin{figure}
    \centering
    \begin{tikzpicture}[scale = 0.8]
    \draw (0,0) -- (0,5);
    \draw (0,0) -- (0.2,0);
    \draw (0,5) -- (0.2,5);
    
    \draw (8, 0) -- (8,5);
    \draw (7.8, 0) -- (8,0);
    \draw (7.8, 5) -- (8,5);
    
    \draw (5.5,0) -- (5.5,5);
    
    \draw (0,4) -- (5.5,4);
    \draw (0,1.5) -- (5.5,1.5);

    \node at (-0.5, 4.5) {$K_1$};
    \node at (-0.5, 3) {$K_2$};
    \node at (-0.5, 1) {$K_3$};

    \node at (2.5, -0.5) {$N_1$};
    \node at (6.75, -0.5) {$N_2$};

    \node at (2.5, 4.5) {$0$};
    \node[scale=0.9] at (2.75, 3.25) {rows can be rescaled s.t.};
    \node[scale=0.9] at (2.75, 2.75) {row $\ell_2$-norm $1$};
    \node[scale=0.9] at (2.75, 2.25) {col $\ell_1$-norm $\le (60\log n)^{-1/2}$};
    \node[scale=0.9] at (2.75, 0.9) {$\ge \ceil{10\log n}$ magnitudes};
    \node[scale=0.9] at (2.75, 0.4) {per row};

    \draw [decorate,decoration={brace,amplitude=5pt}] (0,5.2) -- (5.5,5.2);
    \node at  (2.75, 5.75) {$\ne \emptyset$};
    
    \end{tikzpicture}
    \caption{The decomposition of the coefficient matrix $V$ of the essential cover}
    \label{fig:decomposition}
\end{figure}
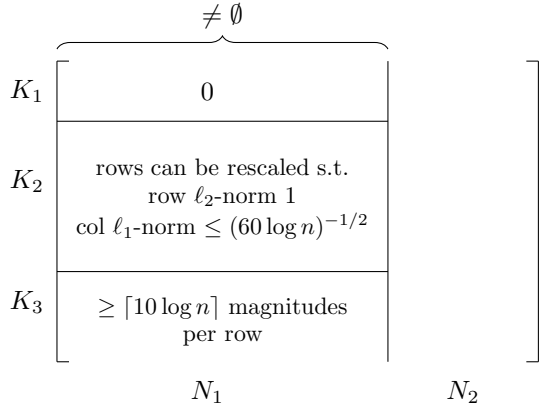

\cref{fig:decomposition} illustrates the decomposition in \cref{prop:decomposition}. We remark that in \cref{prop:decomposition}, some of the subsets $K_1,K_2,K_3\subseteq [k]$ may be empty, and similarly $N_2\subseteq [n]$ may be empty.

Given an essential cover whose coefficient matrix can be decomposed as in \cref{prop:decomposition}, we show that there must exist a vertex in $\signbin^n$ that is not covered by any of the hyperplanes (yielding a contradiction to condition (E1)). This is stated in the following proposition.

\begin{prop}\label{prop:find-vertex}
    For $n\ge 2$, let $h_1,\dots,h_k$ be hyperplanes in $\R^n$ forming an essential cover of the hypercube $\signbin^n$. For each $i\in [k]$, let $h_i=\{x\in \R^n \mid \ang{v_i,x} = \mu_i\}$, and let $V\in \R^{k\times n}$ be the matrix with rows $v_1,\dots,v_k$. Suppose that the matrix $V$ has a decomposition as in \cref{prop:decomposition} and that $k\le n$. Then there exists $w\in \signbin^n$ such that $\ang{v_i, w}\ne \mu_i$ for all $i\in [k]$.
\end{prop}

In other words, this proposition states that actually the coefficient matrix of an essential cover cannot have a decomposition as in \cref{prop:decomposition}. This means that we cannot have $k\le 10^{-2}\cdot n^{2/3}/(\log n)^{2/3}$.

The proof of \cref{prop:find-vertex} employs a similar strategy to \cite{araujo2022,YehudaY21}: we find a certain probability distribution on $\signbin^n$ such that with positive probability a random vertex $w\in \signbin^n$ with this distribution avoids all the hyperplanes. To do so, we first fix the coordinates of $w$ with indices in $N_2$  deterministically using the minimality property of an essential cover (property (E3) in \cref{def:essential-cover}) to always avoid all the hyperplanes $h_i$ with $i\in K_1$. Then we construct probability distributions for the coordinates of $w$ with indices in $N_1$ such that with positive probability $w$ avoids all the hyperplanes $h_i$ with $i\in K_2\cup K_3$. The probability distribution is constructed by applying \cref{cor:bang} (a corollary of Bang's lemma) to the matrix $V'$ in condition (2) in \cref{prop:decomposition}. This way, we find a point $y\in [-1/3,1/3]^{N_1}$ that is ``far away'' from all hyperplanes $h_i$ with $i\in K_2$. We then randomly round $y$ to a vertex in $\signbin^{N_1}$ of the hypercube to define coordinates of $w$ with indices in $N_1$. We can show that for any $i\in K_2$, the vertex $w$ is unlikely to be contained in the hyperplane $h_i$ because $y$ is ``far away'' from $h_i$. To bound the probability that $w$ is contained in a hyperplane $h_i$ with $i\in K_3$, we use that by condition (3) in \cref{prop:decomposition} every row of the submatrix $V[K_3\times N_1]$ has many magnitudes. So for each $i\in K_3$, \cref{lem:many-magnitudes} implies an upper bound for the probability that $w$ is contained in $h_i$.

Despite the similarity in the overall strategy between our argument and the arguments in \cite{araujo2022,YehudaY21}, we highlight the fact that we have a simpler matrix decomposition in \cref{prop:decomposition} and a much shorter proof for \cref{prop:find-vertex}, while obtaining a better bound. The proof of \cref{thm:bound} follows directly from combining \cref{prop:decomposition,prop:find-vertex}. 

\begin{proof}[Proof of \cref{thm:bound}]
Suppose for contradiction that the hyperplanes $h_1,\dots, h_k$ form an essential cover of the hypercube $\signbin^{n}$ of size $k < 10^{-2}\cdot n^{2/3}/(\log n)^{2/3}$. For each $i\in [k]$, let $h_i=\{x\in \R^n \mid \ang{v_i,x} = \mu_i\}$, and let $V\in \R^{k\times n}$ be the matrix with rows $v_1,\dots,v_k$. Then, we can decompose $V$ as in \cref{prop:decomposition}. But now by \cref{prop:find-vertex}, there must be a vertex $w\in \signbin^n$ that is not covered by any of the hyperplanes $h_1,\dots,h_k$. This contradicts condition (E1) in the definition of an essential cover in Definition \ref{def:essential-cover}.
\end{proof}

The rest of the paper is structured as follows. In \cref{sec:find-vertex}, we show that for an essential cover whose coefficient matrix has a decomposition as in \cref{prop:decomposition} there must be a vertex $w\in \signbin^n$ that is not covered by any of the hyperplanes, i.e.\ we prove \cref{prop:find-vertex}. Then in \cref{sec:decompose}, we prove \cref{prop:decomposition}.
\section{Finding the Uncovered vertex}\label{sec:find-vertex}

In this section, we prove \cref{prop:find-vertex}, i.e.\ we show that given an essential cover of the hypercube $\signbin^{n}$ whose the coefficient matrix $V\in \R^{k\times n}$ can be decomposed as in \cref{prop:decomposition}, we can find a vertex $w\in \signbin^n$ that is not covered by any of the hyperplanes. For partitions $[k] = K_1\sqcup K_2\sqcup K_3$ and  $[n] = N_1\sqcup N_2$ as in \cref{prop:decomposition}, we construct $w\in \signbin^n$ by first fixing the entries $w_j$ for $j\in N_2$ to avoid all the hyperplanes $h_i$ with $i\in K_1$ deterministically, and then probabilistically choosing the entries $w_j$ for $j\in N_1$, avoiding all the hyperplanes $h_i$ with $i\in K_2\cup K_3$ with positive probability.

\begin{proof}[Proof of \cref{prop:find-vertex}]
Let $[k] = K_1\sqcup K_2\sqcup K_3$ and  $[n] = N_1\sqcup N_2$ be partitions satisfying the conditions in \cref{prop:decomposition}. Note that these conditions are not affected by rescaling the hyperplane equations $\ang{v_i,x} = \mu_i$ describing the hyperplanes $h_i$ for $i\in K_2$. We may therefore assume that the matrix $V'$ in condition (2) in \cref{prop:decomposition} is simply $V'=V[K_2\times N_1]$ after rescaling the equations $\ang{v_i,x} = \mu_i$ for $i\in K_2$ appropriately (note that this precisely corresponds to rescaling the rows of $V$ with indices in $K_2$). Specifically, this means that we have $||v_i|_{N_1}||_2 = 1$ for each $i\in K_2$, and furthermore, for each $j\in N_1$, the $j$-th column $v_{*j}$ of $V$ satisfies $||v_{*j}|_{K_2}||_1\le (60\log n)^{-1/2}$. 

Next, we claim that $K_1\ne [k]$. Indeed, condition (E2) in \cref{def:essential-cover} means that the matrix $V$ cannot have any column with all entries being zero. On the other hand, by condition (1) in \cref{prop:decomposition} the submatrix $V[K_1\times N_1]$ of $V$ is the zero matrix. Therefore, if we had $K_1= [k]$, then for every $j\in N_1$ the $j$-th column of $V$ would be all-zero. This would be a contradiction, given that $N_1\ne \emptyset$. Thus, we indeed have $K_1\ne [k]$.

Now, we can fix the entries $w_j$ for $j\in N_2$ of our desired vertex $w\in \signbin^n$ using the following lemma.

\begin{lemma}\label{lem:fix-N1}
There exists $\Tilde{w}\in \signbin^{N_2}$ such that for every $w\in \signbin^n$ with $w|_{N_2} = \Tilde{w}$, we have $\ang{v_i, w} \ne \mu_i$ for all $i\in K_1$.  
\end{lemma}

\begin{proof}
By condition (E3) in \cref{def:essential-cover} applied to some element of $[k]\setminus K_1$, there exists a vertex $x\in \signbin^n$ such that $\ang{v_i, x} \ne \mu_i$ for all $i\in K_1$. We set $\Tilde{w} := x|_{N_2}$. By condition (1) in \cref{prop:decomposition}, for every $i\in K_1$ the restriction $v_i|_{N_1}$ is the all-zero vector. Therefore for every $w\in \signbin^n$ with $w|_{N_2} = \Tilde{w}$, we have
\[\ang{v_i, w} = \ang{v_i|_{N_2}, \Tilde{w}} = \ang{v_i|_{N_2}, x|_{N_2}} = \ang{v_i, x}\ne \mu_i\]
for all $i\in K_1$, as desired.
\end{proof}

Let $\Tilde{w}\in \signbin^{N_2}$ be as in \cref{lem:fix-N1}, and fix $w_j = \Tilde{w}_j$ for all $j\in N_2$. Next, we define a probability distribution for the entries $w_j$ for $j\in N_1$. To do so, we apply \cref{cor:bang} to the matrix $V[K_2\times N_1]$ with $\theta = (6\log n)^{1/2}$ and $\mu'\in \R^{K_2}$ defined by $\mu'_i = \mu_i - \ang{v_i|_{N_2}, \Tilde{w}}$ for all $j\in K_2$. We have $\theta ||v_{*j}|_{K_2}||_1\le (6\log n)^{1/2}\cdot (60\log n)^{-1/2}\le 1/3$ for all $j\in N_1$, as well as $||v_i|_{N_1}||_2 = 1$ for all $i\in K_2$, so the conditions of \cref{cor:bang} are satisfied for the matrix $V[K_2\times N_1]$. So we obtain $y\in \R^{N_1}$ such that $||y||_\infty\le 1/3$ and $\labs{\ang{v_i|_{N_1}, y} - \mu_i'} \ge \theta$ for all $i\in K_2$. Then for every $j\in N_1$, let us take a random variable $w_j\in \signbin$ with distribution
\begin{equation}\label{eq:def-w}
    w_j = \begin{cases} 1 & \text{with probability } (1+y_j)/2 \\ -1 & \text{with probability } (1-y_j)/2, \end{cases}
\end{equation}
such that the random variables $w_j$ are independent for all $j\in N_1$. Note that then we have $\E[w_j] = y_j$ for each $j\in N_1$.

Taking the entries of $w$ to be the random variables $w_j$ for $j\in N_1$ together with the fixed values $w_j=\Tilde{w}_j$ for $j\in N_2$ defined above, we have now defined a random vertex $w\in \signbin^n$. Our goal is to show that with positive probability $w\in \signbin^n$ satisfies $\ang{v_i, w} \ne \mu_i$ for all $i\in [k]$. By the choice of $\Tilde{w}$ in \cref{lem:fix-N1} we always have $\ang{v_i, w} \ne \mu_i$ for all $i\in K_1$. Next, we show that for $i\in K_2$ we are unlikely to have $ \ang{w, v_i} = \mu_i$.

\begin{claim}\label{claim:avoid-K2}
For every $i\in K_2$, we have
\[ \Pr \Big[\ang{v_i,w} = \mu_i\Big] \le  \frac{1}{2n}.\]
\end{claim}

\begin{proof}
Fix $i\in K_2$. By our definition of $\mu'_i=\mu_i - \ang{v_i|_{N_2}, \Tilde{w}}=\mu_i - \ang{v_i|_{N_2}, w|_{N_2}}$, we have $\ang{v_i,w} = \mu_i$ if and only if $\ang{v_i|_{N_1},w|_{N_1}} = \mu_i'$. So it suffices to show that 
\[\Pr \Big[\ang{ v_i|_{N_1}, w|_{N_1}} = \mu_i'\Big] \le \frac{1}{2n}.\]

Since $\E[w_j] = y_j$ for all $j\in N_1$, by linearity of expectation we have $\E[\ang{v_i|_{N_1},w|_{N_1}}] = \ang{v_i|_{N_1},y}$. In addition, recall that we have $\labs{\ang{v_i|_{N_1},y} - \mu_i'}\ge \theta$, so we obtain
\begin{align*}
    \Pr\Big[\ang{v_i|_{N_1},w|_{N_1}} = \mu_i'\Big] &\le \Pr\Big[\big|\ang{ v_i|_{N_1},w|_{N_1}} - \ang{v_i|_{N_1},y}\big|\ge \theta \Big] = \Pr\Big[\big|\ang{v_i|_{N_1},w|_{N_1}} - \E[\ang{v_i|_{N_1},w|_{N_1}}]\big|\ge \theta\Big].
\end{align*}
Now, we can apply Hoeffding's inequality (\cref{lem:hoeffding}) to the random variables $v_{ij}w_j$ for $j\in N_1$, whose sum is $\sum_{j\in N_1}v_{ij}w_j=\ang{v_i|_{N_1},w|_{N_1}}$. Note that the random variables $v_{ij}w_j$ are independent and bounded by $-|v_{ij}|\le v_{ij}w_j\le |v_{ij}|$ for each $j\in N_1$. Since we have $||v_i|_{N_1}||_2 = 1$, we obtain
\begin{align*}
    \Pr\Big[\big|\ang{v_i|_{N_1},w|_{N_1}} - \E[\ang{v_i|_{N_1},w|_{N_1}}]\big|\ge \theta\Big] & \le 2\cdot \exp \lpr{-\frac{2\theta^2}{\sum_{j\in N_1} (2v_{ij})^2}}\\
    & = 2\cdot \exp \lpr{-\frac{2\theta^2}{4 ||v_i|_{N_1}||_2^2}}\\
    &=2\cdot \exp \lpr{-\frac{\theta^2}{2}}.
\end{align*}
Substituting in our choice of $\theta = (6\log n)^{1/2}$, we can deduce (recalling $n\ge 2$)
\[\Pr \Big[\ang{v_i|_{N_1},w|_{N_1}} = \mu_i'\Big] \le 2\exp (-3\log n) =  \frac{2}{n^3} \le \frac{1}{2n}.\qedhere\]
\end{proof}

On the other hand, by condition (3) in \cref{prop:decomposition}, for every $i\in K_3$, the vector $v_i|_{N_1}$ has at least $\ceil{10\log n}$ magnitudes. Thus, by \cref{lem:many-magnitudes} applied to $p\in [1/3,2/3]^{N_1}$ given by $p_j=(1+y_j)/2$ for all $j\in N_1$, we have 
\[\Pr\Big[\ang{v_i,w} = \mu_i \Big] =\Pr\Big[ \ang{v_i|_{N_1},w|_{N_1}} = \mu_i -\ang{v_i|_{N_2},\Tilde{w}}\Big]\le \lpr{\frac{2}{3}}^{\ceil{10\log n /2}}\le \lpr{\frac{2}{3}}^{5\log n}\le \frac{1}{n^2}\le \frac{1}{2n}\]
for every $i\in K_3$. Together with \cref{claim:avoid-K2}, we can take a union bound over all $i\in K_2\cup K_3$ and obtain
\[\Pr\Big[\ang{v_i,w} = \mu_i \text{ for some }i\in K_2\cup K_3\Big]\le \sum_{i\in K_2\cup K_3}\Pr\Big[ \ang{v_i,w} = \mu_i \Big]\le k\cdot \frac{1}{2n}\le n\cdot \frac{1}{2n} = \frac{1}{2},\]
recalling that $k\le n$.

Thus, with probability at least $1/2$, our random vertex $w\in \signbin^n$ satisfies $\ang{v_i, w} \ne \mu_i$ for all $i\in K_2\cup K_3$. Recalling that we always have $\ang{v_i, w} \ne \mu_i$ for all $i\in K_1$, this means that with probability at least $1/2$, we have $\ang{v_i, w} \ne \mu_i$ for all $i\in [k]$. Hence there exists a vertex in $w\in\signbin^n$ such that $\ang{v_i, w} \ne \mu_i$ for all $i\in [k]$.
\end{proof}

\section{Decomposing the Matrix}\label{sec:decompose}

Finally, we prove \cref{prop:decomposition}.
\begin{proof}[Proof of \cref{prop:decomposition}]
We decompose the matrix $V$ using the following algorithm that starts with all the column indices in $N_1$. At any iteration of the algorithm, we have a partition $[n] = N_1\sqcup N_2$ of the column indices and furthermore the row indices $i\in [k]$ are partitioned into three types: being in $K_1$, active, and inactive (defined later). Throughout the algorithm, we move certain columns from $N_1$ to $N_2$. When the algorithm terminates, the desired partitions of the row and column indices of the matrix will be formed by taking the sets $K_1$, $N_1$ and $N_2$ at the end of the algorithm, as well as taking $K_2$ to be the active row indices and $K_3$ to be the inactive row indices at the end of the algorithm.

\begin{enumerate}[label=\arabic*.]
    \item Initialize $N_1 = [n],\, N_2= \emptyset$ and $K_1= \emptyset$.

    \item For each $i\in [k]\setminus K_1$ such that $v_i|_{N_1}$ is the all-zero vector, we move $i$ into $K_1$.

    \item For each $i\in [k]\setminus K_1$, we say that $i$ is \emph{inactive} if the vector $v_i|_{N_1}$ has at least $\ceil{10\log n}$ magnitudes, otherwise we say that $i$ is \emph{active}.

    \item For each active $i\in [k]\setminus K_1$ and each $j\in N_1$, we assign a weight $w_{ij}$ to the entry $v_{ij}$ of $V$  as follows: If $v_{ij}=0$, define $w_{ij}=0$. If $v_{ij}\ne 0$, let $m$ be the magnitude of $v_{ij}$ and let $N_m(i)$ be the number of entries of $v_i|_{N_1}$ with magnitude $m$, and define $w_{ij}=10/\sqrt{N_m(i)}$.

    \item  For every $j\in N_1$, calculate the total weight $\sum_i w_{ij}$ in column $j$, where the sum is taken over all active $i\in [k]\setminus K_1$. If there exists some $j\in N_1$ where this total weight is larger than $(60\log n)^{-1/2}$, then move this index $j$ from $N_1$ to $N_2$ and remove all weights in column $j$ (if there exist multiple such $j$, choose one arbitrarily), and go to back to Step 2. If there is no such index $j\in N_1$, go to Step 6.

    \item Terminate and output the current sets $N_1, N_2, K_1$, output the current set of active row indices as $K_2$, and output the current set of inactive row indices as $K_3$. 
\end{enumerate}

It is clear that this algorithm must eventually terminate, since in every iteration a column index is moved from $N_1$ to $N_2$ (so the set $N_1$ is is getting strictly smaller with every iteration). It is also clear that the algorithm produces a partition $[k] = K_1\sqcup K_2\sqcup K_3$ of the row indices of $V$ and a partition $[n] = N_1\sqcup N_2$ of the column indices. Furthermore, it is not hard to see that these partitions satisfy conditions (1) and (3) in \cref{prop:decomposition}. Indeed, condition (1) is satisfied by the rule of moving row indices into $K_1$ in Step 2 of the algorithm, and condition (3) is satisfied by the definition of inactive rows in Step 3.

To show condition (2) in \cref{prop:decomposition}, consider the $|K_2|\times |N_1|$ matrix $V'$ obtained from $V[K_2\times N_1]$ by normalizing every row to have $\ell_2$-norm equal to $1$ (note that by the condition in Step 2, all the rows of $V[K_2\times N_1]$ are non-zero). Let us write $v_i'=(1/||v_i|_{N_1}||_2)v_i|_{N_1}$ for $i\in K_2$ for the rows of $V'$. Now, for every entry $v_{ij}\ne 0$ of $V[K_2\times N_1]$ with some magnitude $m$, we have $10^{m}\le |v_{ij}| < 10^{m+1}$ and for each $i\in K_2$ there are $N_m(i)$ such entries in $v_i|_{N_1}$. Thus, $||v_i|_{N_1}||_2\ge \sqrt{N_m(i)}\cdot 10^{m}$ for every magnitude $m$ appearing among the entries of $v_i|_{N_1}$, and we can conclude that
\[|v'_{ij}|=\frac{|v_{ij}|}{||v_i|_{N_1}||_2}< \frac{10^{m+1}}{\sqrt{N_m(i)}\cdot 10^{m}} = \frac{10}{\sqrt{N_m(i)}} = w_{ij}\]
for every entry $v_{ij}\ne 0$ of $V[K_2\times N_1]$ with magnitude $m$. This shows that $|v_{ij}'| \le w_{ij}$ for all $i\in K_2$ and $j\in N_1$ (note that in the case $v_{ij}= 0$, we trivially have $|v_{ij}'|=0 = w_{ij}$).

Now, to check condition (2), observe that for every $j\in N_1$ we have $\sum_{i\in K_2} w_{ij}\le (60\log n)^{-1/2}$ since otherwise the algorithm would not have terminated. Consequently, for every $j\in N_1$ we have $\sum_{i\in K_2} |v_{ij}'|\le \sum_{i\in K_2} w_{ij}\le (60\log n)^{-1/2}$, meaning that the $\ell_1$-norm of every column in $V'$ is at most $(60\log n)^{-1/2}$.

It remains to show $N_1\ne \emptyset$. Recall that at the start we have $N_1=[n]$, and in every iteration of the algorithm we move one index from $N_1$ into $N_2$. So we need to show that the number of column indices moved into $N_2$ throughout the algorithm is less than $n$. At every step, we consider the total weight in the matrix. Whenever we move a column index from $N_1$ into $N_2$, this total weight decreases by at least $(60\log n)^{-1/2}$ (by the condition on the column weight in Step 5 of the algorithm).

However, note after moving a column from $N_1$ to $N_2$, when reassigning the weights in Step 3 in the next iteration, some of the weights in the remaining columns with indices in $N_1$ may increase (indeed, the numbers $N_m(i)$ may decrease by removing a column from $N_1$). Furthermore, new row indices  may become active as the set $N_1$ gets smaller throughout the algorithm (but note that once a row index becomes active, it cannot become inactive anymore afterwards).

For every row index $i\in [k]$ that becomes active at some point during the algorithm, and every magnitude occurring among the entries of $v_i|_{N_1}$ at the moment when the index $i$ first becomes active, let $N^*_m(i)$ denote the number of entries of $v_i|_{N_1}$ with magnitude $m$ when the row index $i$ first becomes active. At that iteration of the algorithm, each of these $N^*_m(i)$ entries is assigned weight $10/\sqrt{N^*_m(i)}$. When a column index of some entry of $v_i|_{N_1}$ with magnitude $m$ gets moved into $N_2$ later in the algorithm, the quantity $N_m(i)$ decreases and so the weights of the remaining entries of $v_i|_{N_1}$ with magnitude $m$ increase. More precisely, when this happens for the $t$-th time (for $t\in \{1,\dots,N^*_m(i)-1\}$), the weight of the $N^*_m(i)-t$ remaining entries of $v_i|_{N_1}$ with magnitude $m$ is increased from $10/\sqrt{N^*_m(i)-t+1}$ to $10/\sqrt{N^*_m(i)-t}$. Thus, the total weight that is added to the entries with magnitude $m$ in row $i$ throughout the algorithm (both when row $i$ first becomes active and as the weights get re-assigned) is at most
\begin{align*}
    & N^*_m(i)\cdot \frac{10}{\sqrt{N^*_m(i)}} + \sum_{t = 1}^{N^*_m(i)-1} \lpr{N^*_m(i) - t} \lpr{\frac{10}{\sqrt{N^*_m(i) - t}} - \frac{10}{\sqrt{N^*_m(i) - t + 1}}}\\
     &=   10\sqrt{N^*_m(i)} + 10\sum_{t = 1}^{N^*_m(i) - 1} t \lpr{\frac{1}{\sqrt{t}} - \frac{1}{\sqrt{t+1}}} 
     =  10\sqrt{N^*_m(i)} + 10\sum_{t = 1}^{N^*_m(i) - 1} t \cdot \frac{\sqrt{t+1} - \sqrt{t}}{\sqrt{t(t+1)}} \\
     &\le  10\sqrt{N^*_m(i)} + 10\sum_{t = 1}^{N^*_m(i) - 1} \lpr{\sqrt{t+1} - \sqrt{t}} \le 10\sqrt{N^*_m(i)} +  10\sqrt{N^*_m(i)} =  20\sqrt{N^*_m(i)}
\end{align*}
for any row index $i$ becoming active at some point during the algorithm and any $m$ appearing as magnitude among the entries of $v_i|_{N_1}$ at the moment when the index $i$ first becomes active.

For notational convenience, let $S = \ceil{10\log n}-1$. Recall that when the row index $i$ becomes active, there are strictly less than $\ceil{10\log n}$ (i.e.\ at most $S$) magnitudes appearing among the entries of $v_i|_{N_1}$. For these at most $S$ different magnitudes $m$, the numbers $N^*_m(i)$ defined above (counting the entries of $v_i|_{N_1}$ of magnitude $m$ when the row index $i$ first becomes active) satisfy
\[\sum_m N^*_m(i) = |\supp(v_i|_{N_1})|\le |\supp(v_i)|\le 2k,\]
where the last inequality holds by \cref{lem:row-support}.
Hence
\[\sum_m 20 \sqrt{N_m(i)}= 20\sum_m  \sqrt{N_m(i)}\le 20 \cdot\sqrt{S} \cdot \sqrt{\sum_m N^*_m(i)}\le 20 \cdot\sqrt{S} \cdot\sqrt{2k} \le 30 \sqrt{kS}\]
by the Cauchy--Schwarz inequality, so for every row the total weight added to the entries throughout the algorithm is at most $30 \sqrt{kS}$. Thus, the total weight increase for the entire $k\times n$ matrix $V$ throughout the algorithm is at most
\[k\cdot 30 \sqrt{kS}=30k^{3/2}\sqrt{S}\le 30k^{3/2}\cdot (10\log n)^{1/2}.\]
Recall that whenever we move a column index from $N_1$ into $N_2$, the weight in the matrix $V$ decreases by at least $(60\log n)^{-1/2}$. Therefore, at most
\[\frac{30k^{3/2}\cdot (10\log n)^{1/2}}{(60\log n)^{-1/2}}\le 750k^{3/2}\cdot \log n\]
column indices can be moved from $N_1$ into $N_2$. By our assumption $k \le 10^{-2}\cdot n^{2/3}/(\log n)^{2/3}$, this number of column indices moved from $N_1$ into $N_2$ is at most
\[750k^{3/2}\cdot \log n\le 750\cdot 10^{-3}\cdot \frac{n}{\log n}\cdot \log n < n.\]
Thus, we have $N_1\ne \emptyset$ at the end of the algorithm, as desired.
\end{proof}

\printbibliography

\end{document}